\documentclass[12pt, article]{article}
\usepackage{geometry} 
\usepackage{ae} 
\usepackage[T1]{fontenc}
\usepackage[cp1250]{inputenc}
\usepackage{amsmath}
\usepackage{dsfont}
\usepackage{amssymb, amsfonts,amscd,verbatim}
\usepackage{MnSymbol}
\usepackage{mathrsfs} 
\usepackage{amsthm}
\usepackage{setspace}   
\usepackage{lipsum} 
\usepackage[normalem]{ulem}
\usepackage{hyperref}
\usepackage{indentfirst}
\usepackage{latexsym}
\input xy
\xyoption{all}

\newtheorem{Pocz}{Poczatek}[section]
\newtheorem{prop}[Pocz]{Proposition}

\newtheorem{Theorem}[Pocz]{Theorem}
\newtheorem{Corollary}[Pocz]{Corollary}

\newtheorem{Lemma}[Pocz]{Lemma}
\newtheorem{Remark}[Pocz]{Remark}
\newtheorem{Notation}[Pocz]{Notation}
\newtheorem{Question}[Pocz]{Question}

\newtheorem{Example}[Pocz]{Example}
\newtheorem{Definition}{Definition}

\newcommand{\UU}{\mathcal{U}}
\newcommand{\VV}{\mathcal{V}}
\newcommand{\WW}{\mathcal{W}}
\DeclareMathOperator*{\st}{st}

\begin{document}
\begin{figure}
\centering
{\large Bounded Scale Measure and Property A} 

\vspace{0.2 cm}

Kevin Sinclair and Logan Higginbotham
\date{ \today}
\end{figure}
\begin{abstract}
We introduce a generalization for bounded geometry that we call bounded scale measure. We show that bounded scale measure is a coarse invariant unlike bounded geometry. We then show equivalent definitions for spaces with bounded scale measure and show other properties that spaces with bounded scale measure satisfy. From there, we generalize property A from uniformly discrete metric spaces to large scale spaces with bounded geometry. Lastly, we construct a definition for property A for large scale spaces with bounded scale measure and show that this definition of property A is a coarse invariant.
\end{abstract}

\section*{Acknowledgement}
We would like to thank Campbell University's faculty development and research committee for providing funding to help with this paper.

\section{Introduction}

The main focus of this paper is to introduce a generalization of bounded geometry that we will call bounded scale measure. From there, we will introduce several useful equivalences of bounded scale measure and properties that spaces with bounded scale measure satisfy. Finally, we will define a generalization of property A for large scale structures with bounded geometry and for large scale spaces with bounded scale measure. We call this new generalization property A at scale $\UU$, and show that it is a coarse invariant. To begin, we will introduce the basic definitions needed to understand large scale structures.

\begin{Definition}
	Let $\mathcal{U}$ be a family of subsets of a set $X$ and let $V$ be a subset of $X$. The \textbf{star} of $V$ against $\mathcal{U}$, denoted $\st(V,\mathcal{U})$, is the set 
	$$\bigcup\limits_{\substack{U\in \mathcal{U}\\{U\cap V\neq \varnothing}}} U$$ 
	If $\mathcal{V}$ is another family of subsets of $X$, then the family $\left\{\st(V,\mathcal{U})|V\in \mathcal{V}\right\}$ of subsets of $X$ is denoted $\st(\mathcal{V},\mathcal{U})$ for convenience.
\end{Definition}

\begin{Definition}
	Let $\mathcal{U},\mathcal{V}$ be families of subsets of a set $X$. We say $\mathcal{U}$ is a \textbf{refinement} of $\mathcal{V}$ provided for every $U\in\mathcal{U}$ there is a $V\in\mathcal{V}$ so that $U\subseteq V$. In this same situation, we also say that $\mathcal{V}$ \textbf{coarsens} $\mathcal{U}$. Refinement is denoted as $\mathcal{U}\prec\mathcal{V}$.
\end{Definition}

It is sometimes needed to consider covers of $X$ instead of collections of subsets of $X$. To distinguish families of subsets of $X$ from covers of $X$, we call covers of $X$ scales:

\begin{Definition}
	Given a set $X$, we say $~\mathcal{U}$ is a \textbf{scale} of $X$ if $~\mathcal{U}$ is a family of subsets of $X$ that covers $X$. If $~\mathcal{U}$ is a collection of subsets of $X$, we can make $\mathcal{U}$ into a cover via constructing $\mathcal{U}'=\mathcal{U}\cup \left\{\{x\}\right\}_{x\in X}$. This extension is often called the \textbf{trivial extension} of $\mathcal{U}$.
\end{Definition}

The definition of large scale structures was given by Dydak and Hoffland in \cite{Dydak}. This equivalent interpretation of coarse structures allows coarse geometry to be approached in a more topological way.

\begin{Definition}
	Let $X$ be a set. A \textbf{large scale structure} on $X$ is a non-empty set of families $\mathcal{LSS}$ of subsets of $X$ so that the following conditions are satisfied:
	\begin{enumerate}
		\item If we have families $\mathcal{U},\mathcal{V}$ of subsets of X with $\mathcal{V}\in\mathcal{LSS}$ and each element $U$ of $\mathcal{U}$ consisting of more than one point is contained in some $V$ of $\mathcal{V}$, then $\mathcal{V}\in\mathcal{LSS}$.
		\item If $\mathcal{U},\mathcal{V}\in\mathcal{LSS},$ then $\st(\mathcal{U},\mathcal{V})\in\mathcal{LSS}$.
	\end{enumerate} 
	Elements  $~\mathcal{U}$ of $\mathcal{LSS}$ are called \textbf{uniformly bounded families} or \textbf{uniformly bounded scales}.
If a large scale structure exists on a space $X$, we call the pair $(X, \mathcal{LSS})$ a \textbf{large scale space}.
\end{Definition} 

We note here closure under refinements implies the first condition above. The advantage of having a weaker first requirement is that a large scale structure as defined "disregards" one point sets. That is, one point sets do not "change" the large scale structure. Also, the first item in the definition gives us that the cover $\{\{x\}\}_{x\in X}$ is uniformly bounded for any large scale structure. 

The following definition is also a nice example of large scale structures:

\begin{Definition}
Let $M$ be a metric space with metric $d$. $\mathcal{LSS}$ is the \textbf{large scale structure induced by the metric} is the following collection of uniformly bounded families: $\UU\in\mathcal{LSS}$ if and only if $\sup\limits_{U\in\UU} diam(U)$ is finite, where $diam(U)=\sup\limits_{x,y\in U} d(x,y)$. Note that this definition also applies to metrics that are allowed to take on the value $\infty$. Such metrics are called $\infty$-\textbf{metrics}.
\end{Definition}

\begin{Remark}
In \cite{Dydak}, it was shown in Theorem 1.8 that a large scale structure induced by a metric is generated by the covers by $n-balls$ for $n\in\mathds{N}$. This is to say that every uniformly bounded family of a large scale structure is a refinement of a cover of $n-balls$ for some $n$. We will use this idea occasionally throughout the paper.
\end{Remark}

As is often the case with mathematical objects, we seek to compare one object to another. We do this through large scale continuous functions. Think of large scale continuous functions as analogous to continuous functions in topology.

\begin{Definition}
	Let $(X,\mathcal{LSS}_X)$ and $(Y,\mathcal{LSS}_Y)$ be large scale spaces and let $f:X\to Y$. We say $f$ is \textbf{large scale continuous} or \textbf{bornologous} if for every\\ $\mathcal{U}\in\mathcal{LSS}_X,~f(\mathcal{U})\in\mathcal{LSS}_Y$, where $f(\mathcal{U})=\{f(U)|~U\in\mathcal{U}\}$. 
\end{Definition}

There are other types of large scale functions that are useful to know:

\begin{Definition}
	Let $(X,\mathcal{LSS}_X)$ and $(Y,\mathcal{LSS}_Y)$ be large scale spaces and let $f:X\to Y$. We say $f$ is a \textbf{coarse embedding} if for every $\mathcal{V}\in\mathcal{LSS}_Y$, we have $f^{-1}(\mathcal{V})\in\mathcal{LSS}_X$, where $f^{-1}(\mathcal{V})=\{f^{-1}(V)|~V \in \mathcal{V}\}$. We say $f$ is \textbf{coarsely surjective} if there exists a $\mathcal{V}\in\mathcal{LSS}_Y$ so that $Y\subseteq\st(f(X),\mathcal{V})$.
\end{Definition}

We now wish to define something analogous to homeomorphisms in topology. However, there is a wrinkle. Homeomorphisms require functional inverses (i.e. the composition of $f$ and its inverse in either order is the identity function on the appropriate set). However, in large scale geometry we can weaken the notion of requiring a functional inverse. This is because two distinct functions can appear similar should one ``zoom out'' far enough. That is, the outputs of the functions remain a uniformly bounded distance apart. Such functions are called close. 
 
\begin{Definition}
	Let $(X,\mathcal{LSS}_X)$ and $(Y,\mathcal{LSS}_Y)$ be large scale spaces and let \\ $f,g:X\to Y$. We say $f$ and $g$ are \textbf{close} provided there is a $\mathcal{V}\in\mathcal{LSS}_Y$ so that for any $x\in X,~f(x),g(x)\in V$ for some $V\in\mathcal{V}$.
\end{Definition}


Notice how the following definition of our version of a homeomorphism differs only in that we require our functional compositions are close to the appropriate identity as opposed to being the identity functions themselves:

\begin{Definition}
	Let $(X,\mathcal{LSS}_X)$ and $(Y,\mathcal{LSS}_Y)$ be large scale spaces and let $f:X\to Y$ be large scale continuous. f is a \textbf{coarse equivalence} if and only if there exists a large scale continuous map $g:Y\to X$ so that $f\circ g$ is close to $id_Y$ and $g\circ f$ is close to $id_X$.
\end{Definition}

\begin{Remark}
Let $(X,\mathcal{LSS}_X)$ and $(Y,\mathcal{LSS}_Y)$ be large scale spaces. We show in the appendix that $f:X\to Y$ is a coarse equivalence if and only if $f:X\to Y$ is a coarse embedding and coarsely surjective.
\end{Remark}

\begin{Example}
Let $(\mathbb{Z},\mathcal{LSS})$ and $(\mathbb{R},\mathcal{LSS}')$ be the large scale structures induced by the metric of absolute value. These large scale structures are coarsely equivalent. Let $f:\mathbb{Z}\to\mathbb{R}$ be inclusion and let $g:\mathbb{R}\to\mathbb{Z}$ be the floor function. One can show that $f$ and $g$ are large scale continuous and are close to the appropriate identity functions when composed.
\end{Example}

Like in topology, we are interested in studying properties of large scale spaces that are preserved by coarse equivalences. We give a definition of such properties below:

\begin{Definition}
	Let $(X,\mathcal{LSS}_X)$ be a large scale space. A property $P$ is a \textbf{coarse invariant} if for any $(Y,\mathcal{LSS}_Y)$ that has property $P$ and is coarsely equivalent to $(X,\mathcal{LSS}_X)$ we have that $(X,\mathcal{LSS}_X)$ also has property $P$.
\end{Definition}
\section{Bounded Scale Measure}

In order to motivate bounded scale measure, recall the definition of a large scale space having bounded geometry:

\begin{Definition}
Let $(X,\mathcal{LSS})$ be a large scale space. We say $(X,\mathcal{LSS})$ has \textbf{bounded geometry} if for every $\UU\in\mathcal{LSS}$ there exists an $n\in\mathds{N}$ so that for all $U\in\UU$ we have that $|U|\leq n$.
\end{Definition}

One of the main issues with considering large scale spaces with bounded geometry is that bounded geometry is not a coarse invariant. Indeed, as in Example 1.2 $(\mathds{Z},\mathcal{LSS})$ and $(\mathds{R},\mathcal{LSS}')$ are coarsely equivalent, but $(\mathds{Z},\mathcal{LSS})$ has bounded geometry while $(\mathds{R},\mathcal{LSS}')$ does not. The idea for bounded scale measure came while trying to create a generalization of bounded geometry that is also a coarse invariant. The idea is to consider elements of a special scale, normally referred to as $\UU$, to be the ``points" of the set, and use that to get a finite bounds on all other scales of $X$. 

We will now give a definition of bounded scale measure with regards to maximal sets. From there, we will give two equivalent definitions that are useful for proving other properties that spaces with bounded scale measure satisfy. Finally, we will show that having bounded scale measure is a coarse invariant.

\begin{Definition} Given a scale $\mathcal{U}$ on $X$ and given a subset $A$ of $X$, a \textbf{$\mathcal{U}$-net} in $A$ is a maximal subset $B$ of $A$ with respect to the following property:
no two elements of $B$ belong to the same element of the scale $\mathcal{U}$.
\end{Definition}

\begin{Definition}
$(X,\mathcal{LSS})$ has \textbf{bounded scale measure at scale} $\UU$ if there exists a scale $\mathcal{U}$ such that for all scales $\mathcal{V}$ there exists a natural number $u(\mathcal{V})$ such that for all $V\in \mathcal{V}$ any $\mathcal{U}$-net of $V$
has at most $u(\mathcal{V})$ elements.
\end{Definition}

It turns out that this definition can be weakened somewhat. Having bounded scale measure at scale $\UU$ currently requires that for any scale $\VV$ and any element $V$ of $\VV$ has the property described \textbf{for all} $\UU$-nets of $V$. However, we will show that having bounded scale measure at scale $\UU$ is equivalent to having a scale $\mathcal{W}$ so that for any scale $\VV$ and element $V$ of $\VV$, \textbf{there exists} a $\mathcal{W}$-net of $V$ that has the property described. Namely, the scale $\mathcal{W}$ is $\st(\UU,\UU)$.

\begin{Theorem}
$(X,\mathcal{LSS})$ has bounded scale measure if and only if there exists a scale $\mathcal{U}$ such that for all scales $\mathcal{V}$ there exists a natural number $n(\mathcal{V})$ so that for all $V\in \mathcal{V}$ there is a $\mathcal{U}$-net of $V$
which has at most $n(\mathcal{V})$ elements.
\end{Theorem}
\begin{proof}
First suppose there exists a scale $\mathcal{U}$ such that for all scales $\mathcal{V}$ there exists a natural number $u(\mathcal{V})$ such that for all $V\in \mathcal{V}$ any $\mathcal{U}$-net of $V$
has at most $u(\mathcal{V})$ elements. Let $\mathcal{V}$ be a scale of $X$ and let $V\in \mathcal{V}$. Take $x\in V$ and then choose all points $x_i, x_j \in X$  such that if $i\neq j$ $\{x_i,x_j\}\notin U$ for all $U\in \mathcal{U}$. By construction this is a $\mathcal{U}$-net in $V$ and thus there exists $k \in \mathbb{N}$ with $k\leq u(\mathcal{V})$, as otherwise we have a $\mathcal{U}$-net with more than $u(\mathcal{V})$ elements. Thus taking our $n(\mathcal{V})=u(\mathcal{V})$ and the collection $\{x_i\}_{1\leq i \leq k}$ satisfies our condition.\\

Conversely, suppose there exists a scale $\mathcal{U}$ such that for all scales $\mathcal{V}$ there exists a natural number $n(\mathcal{V})$ such that for all $V\in \mathcal{V}$ there is a $\mathcal{U}$-net, $B$, which has at most $n(\mathcal{V})$ elements. Define $\mathcal{W}=st(\mathcal{U},\mathcal{U})$, and let $\mathcal{V}$ be any scale of $X$. Let $A$ be any $\mathcal{W}$-net in $V$ and $B$ be the $\mathcal{U}$-net with $|B|\leq n(\mathcal{V})$. 

We claim that $|A\setminus B| \leq n({\mathcal{V}})$. By contradiction assume $|A\setminus B|>n(\mathcal{V})$. By maximality of $B$, given any $x_i\in A\setminus B$, there exists $y_i \in B$ such that $\{x_i,y_i\}\subset U_i$ for $U_i\in \mathcal{U}$. As well, for any $x_j \in A \setminus B$ with $x_i\neq x_j$ we can find a $y_j$ with the same property above. If  $y_i=y_j$, then $\{x_i,x_j\}\subset st(U_i,\mathcal{U})$ which is a contradiction as $A$ is a $\mathcal{W}$-net in $V$. Now, take a collection of elements $\{x_i\}_{1\leq i \leq n(\mathcal{V})}$ of $A\setminus B$, and associate with each $x_i$ the unique $y_i \in B$ and form the collection $\{y_i\}_{1\leq i \leq n(\mathcal{V})}$. $\{y_i\}_{1\leq i \leq n(\mathcal{V})} = B$ as each $y_i$ is unique and $|B|\leq n(\mathcal{V})$. Since $|A\setminus B|>n(\mathcal{V})$ there exists $x \in (A\setminus B)\setminus \{x_i\}_{1\leq i\leq n(\mathcal{V})}$. However, by construction, given any $y\in B$, $\{x,y\}\notin U$ for all all $U \in \mathcal{U}$. Otherwise, suppose there exists a $y\in B$ with $\{x,y\}\subset U$. Since $\{y_i\}_{1\leq i \leq n(\mathcal{V})} = B$, $y=y_i$ for some $i$ and $\{x,x_i\}\subset st(U_i,\mathcal{U})$ which contradicts $A$ being a $\mathcal{W}$-net. Thus $\{x,y\}\notin U$ for all all $U \in \mathcal{U}$. Therefore $B$ is not maximal as $x\notin B$. Hence, $B$ is not a $\mathcal{U}$-net which gives a contradiction. Thus $|A\setminus B| \leq n({\mathcal{V}})$, and any $\mathcal{W}$-net must have less than $2n(\mathcal{V})+1$ elements as $\mathcal{W} = st(\mathcal{U},\mathcal{U})$. This satisfies the definition of bounded scale measure at scale $\mathcal{W}$.
\end{proof}
Our next definition is useful for working with scales as it gives a way to compare any scale $\VV$ against the scale $\UU$ where the $\UU$-nets are defined.
\begin{Theorem}
Given a scale $\UU$, a large scale space $(X,\mathcal{LSS})$ has bounded scale measure at scale $\mathcal{U}$ if and only if for all scales $\mathcal{V}$ there exists a natural number $k(\mathcal{V})$ such that for all $V \in \mathcal{V}$, $V\subseteq \bigcup\limits_{i=1}^{k(\mathcal{V})} U_i$ for a collection $U_i\in \mathcal{U}$.
\end{Theorem}

\begin{proof}
Let $(X,\mathcal{LSS})$ have bounded scale measure at scale $\UU$, and let $\mathcal{V}$ be a scale of $\mathcal{LSS}$. Let $\mathcal{U}$ and $u(\mathcal{V})$ be defined as they are in the definition of bounded scale measure. Define $\mathcal{W} = st(\mathcal{U},\mathcal{U})$. Take a $\mathcal{U}$-net of $V \in \mathcal{V}$ and call it $B$. Then $|B| \leq u(\mathcal{V})$ since  $(X,\mathcal{LSS})$ has bounded scale measure. As $\UU$ is a cover of $X$, for each $b\in B$ select $U_b \in \mathcal{U}$ such that $b\in U_b$, and define $W_b = st(U_b,\mathcal{U})$.

First, we claim that $V \subseteq \bigcup\limits_{i=1}^{u(\mathcal{V})}W_{b_i}$.
Let $x\in V$. Then if $x \in B$, by construction there exists a $U_x \in \mathcal{U}$ such that $x \in U_x \subseteq \bigcup\limits_{i=1}^{u(\mathcal{V})}W_{b_i}$. If $x \notin B$ then by maximality of $B$ there exists $b\in B$ such that $\{x,b\}\subset U$ for some $U\in \mathcal{U}$. This means $x\in st(U_b,\mathcal{U})$ and $x\in  \bigcup\limits_{i=1}^{u(\mathcal{V})}W_{b_i}$. Therefore, $V \subseteq \bigcup\limits_{i=1}^{u(\mathcal{V})}W_{b_i}$ Letting $\mathcal{W}$ be our designated scale and letting $k(\mathcal{V})=u(\mathcal{V})$ satisfies our condition.\\

Conversely, suppose there exists a scale $\mathcal{U}$ such that for all scales $\mathcal{V}$ there exists a natural number $k(\mathcal{V})$ such that for all $V \in \mathcal{V}$, $V\subset \bigcup\limits_{i=1}^{k(\mathcal{V})} U_i$ for a collection $U_i\in \mathcal{U}$. Let $\mathcal{V}$ be any element of $\mathcal{LSS}$. Let $V\in \mathcal{V}$ and let $B$ be a $\mathcal{U}$-net of $V$. \\
We claim that $|B|\leq k(\mathcal{V})$. By contradiction assume $|B|>k(\mathcal{V})$. Then by hypothesis there exist $U_i\in \mathcal{U}$ for $1\leq i\leq k(\VV)$ such that $V\subseteq \bigcup\limits_{i=1}^{k(\mathcal{V})} U_i$. Given any $x,y\in B$, then $\{x,y\}\nsubseteq U_i$ for all $1\leq i \leq  k(\mathcal{V})$. However, each $x\in B$ must be in some $U_i$ as the $U_i$'s cover $V$. For each $U_i$, assign the unique $x_i\in B$ such that $x_i \in U_i$. By assumption $|B|>k(\mathcal{V})$, so there exists $x\in B\setminus\{x_i\}_{1\leq i \leq k(\mathcal{V})}$. Since $B$ is a $\mathcal{U}$-net, then $x \notin U_i$ for all $i$ with $1\leq i \leq k(\mathcal{V})$. However this means $x\notin V$ as $V\subseteq \bigcup\limits_{i=1}^{k(\mathcal{V})} U_i$. This tells us that $B \nsubseteq V$, so $B$ is not a $\mathcal{U}$-net which is a contradiction. Therefore for any $\mathcal{U}$-net in $V$, $|B|\leq k(\mathcal{V})$. Then by definition, $(X,\mathcal{LSS})$ has bounded scale measure at scale $\UU$.
\end{proof}

\begin{Corollary}
Let $(X, \mathcal{LSS})$ have bounded scale measure at scale $\UU$ and let $\UU\prec\VV$ with $\VV\in\mathcal{LSS}$. Then $(X,\mathcal{LSS})$ has bounded scale measure at scale $\VV$.
\end{Corollary}

\begin{proof}
Let $(X, \mathcal{LSS})$ have bounded scale measure at scale $\UU$ and let $\UU\prec\VV$ with $\VV\in\mathcal{LSS}$. Let $\WW$ be any other scale in $\mathcal{LSS}$. Then by Theorem 2.2 any element $W$ of $\WW$ can be covered by $n(\UU)$ elements of $\UU$, call them $U_1,...,U_{n(\UU)}$. As $\UU\prec\VV$, each $U_1,...,U_{n(\UU)}$ is contained in an element of $\VV$; call them $V_1,...,V_{n(\UU)}$. Since we have for any $W\in\WW$ $W\subseteq\bigcup\limits_{i=1}^{n(\UU)} U_i$ and $\bigcup\limits_{i=1}^{n(\UU)} U_i\subseteq\bigcup\limits_{i=1}^{n(\UU)} V_i$, we have any element of $\WW$ can be covered by finitely many elements of $\VV$. Thus, $(X,\mathcal{LSS})$ has bounded scale measure at scale $\VV$ by Theorem 2.2.  
\end{proof}

The easiest example of large scale spaces with bounded scale measure is large scale spaces with bounded geometry. After all, bounded scale measure is meant to generalize bounded geometry. Bounded geometry is thought to be bounded scale measure at scale $\UU$ where $\UU$ is the cover of one point sets.

\begin{prop}
Let $(X,\mathcal{LSS})$ be a large scale space. If $(X,\mathcal{LSS})$ has bounded geometry, then $(X,\mathcal{LSS})$ has bounded scale measure.
\end{prop}

\begin{proof}
Define $\mathcal{U}$ to be the trivial cover of $X$. Let $\mathcal{V}$ be any scale of $X$. Since $(X,\mathcal{LSS})$ has bounded geometry, then for any $V \in \mathcal{V}$ there is an $n\in\mathds{N}$ such that $|V| \leq n$. Therefore, $V=\displaystyle\bigcup_{i=1}^{n}\{x_i\}$, and letting $k(\mathcal{V})=n$ gives us $(X,\mathcal{LSS})$ has bounded scale measure.
\end{proof}

It turns out that if $(X,\mathcal{LSS})$ has bounded scale measure and $Y\subseteq X$, then $Y$ will have bounded scale measure too provided that we equip $Y$ with the appropriate large scale structure:

\begin{Definition}
Let $(X,\mathcal{LSS})$ be a large scale space and $Y\subseteq X$. The \textbf{subspace large scale structure} $\mathcal{LSS}'$ of $Y$ is the collection of uniformly bounded families defined in the following way: $\UU '\in\mathcal{LSS}'$ provided there exists $\UU\in\mathcal{LSS}$ such that $\UU '=\{U\cap Y~|~U\in\UU\}$.
\end{Definition}

\begin{prop} Let $(X,\mathcal{LSS})$ be a space with bounded scale measure at scale $\UU$. Then any subset $Y$ of $X$  with the subspace large scale structure also has bounded scale measure.
\end{prop}
\begin{proof}
Let $(X,\mathcal{LSS})$ be a space with bounded scale measure at scale $\mathcal{U}$, and let $Y\subseteq X$ with $(Y,\mathcal{LSS}')$ the subspace large scale structure of $(X,\mathcal{LSS})$.Define $\UU '\in\mathcal{LSS}'$ to be $\UU '=\{U\cap Y~|~U\in\UU\}$. Let $\mathcal{V}'$ be a scale of $\mathcal{LSS}'$ and for $V'\in \mathcal{V}'$, let $B_{V'}$ be a $\mathcal{U'}$-net of $V'$. Define $\mathcal{V''}$ to be the trivial extension of $\mathcal{V'}$ to a scale of $\mathcal{LSS}'$ and $u(\mathcal{V''})$ be the natural number from the definition of bounded scale measure for $X$. If $V'' \subset X\setminus Y$, then $V''$ must be a point and a $\mathcal{U}$-net must also be a point. Otherwise $V''$ is an element of $\mathcal{V'}$, and a $\mathcal{U}$-net of $V''$ is also a $\mathcal{U'}$-net of $V'$. This is because given any $x,y\in B_{V'}\subset Y$, we know that for all $U\in \mathcal{U}, \{x,y\} \nsubseteq U \cap Y=U'$. As $x,y\in Y$, we then have $\{x,y\} \nsubseteq U$ for all $U\in \mathcal{U}$. Therefore, for any scale $\mathcal{V'}$ of $Y$, a $\mathcal{U'}$-net $B$ is a $\mathcal{U}$-net and since $(X,\mathcal{LSS})$ has bounded scale measure, $|B| \leq u(\mathcal{V''})$. Thus for any scale $\mathcal{V}'$ of $\mathcal{LSS}'$, setting $u(\mathcal{V})=u(\mathcal{V''})$, where $\mathcal{V''}$ is the trivial extension of $\mathcal{V'}$ over $X$ shows $(Y,\mathcal{LSS}')$ has bounded scale measure at scale $\UU'$.
\end{proof}

The following is a key result that distinguishes bounded scale measure from bounded geometry. A large scale space having bounded scale measure at some scale is a coarse invariant:

\begin{Theorem}
Let $(X,\mathcal{LSS}_X)$ and $(Y,\mathcal{LSS}_Y)$ be large scale spaces and let $f:X\rightarrow Y$ be a coarse equivalence. Then $(X,\mathcal{LSS}_X)$ has bounded scale measure if and only if $(Y,\mathcal{LSS}_Y)$ has bounded scale measure.
\end{Theorem}
\begin{proof}
Let $(X,\mathcal{LSS}_X)$ be a space with bounded scale measure at scale $\UU$. Let\\ $f:X\rightarrow Y$ be a coarse equivalence. Then $f(\mathcal{U})$ is a scale in $(Y,\mathcal{LSS}_Y)$. Now, let $\mathcal{V}$ be any scale of $(Y,\mathcal{LSS}_Y)$. For arbitrary $V\in \mathcal{V}$ let $B_V$ be a $f(\mathcal{U})$-net of $V$. We know that $f^{-1}(\mathcal{V})$ is a scale in $(X,\mathcal{LSS}_X)$ since $f$ is a coarse equivalence. We will show $f^{-1}(B_V)$ is a $\mathcal{U}$-net of $f^{-1}(V)$. By contradiction assume there exists $x,y\in f^{-1}(B_V)$ such that there is some $U\in \mathcal{U}$ with $x,y\in U$. However, this implies $f(x),f(y)\in f(U)$, but  $x,y\in f^{-1}(B_V)$ which implies $f(x),f(y)\in B_V$ which contradicts that $B_V$ is a $f(\mathcal{U})$-net. We now show maximality. Assume by contradiction that there is some $z\in f^{-1}(V)\setminus f^{-1}(B_V)$ such that for all $x \in B_V$ there does not exist a $U\in \mathcal{U}$ with $x,z\in U$. However, this implies there is not $f(U)\in f(\mathcal{U})$ with $f(x),f(z) \in f(U)$. By maximality of $B_V$, $f(z)\in B$ which contradicts that fact that $z \notin f^{-1}(B_V)$. Therefore $f^{-1}(B_V) $ is a $\mathcal{U}$-net of $f^{-1}(V)$. Since $(X,\mathcal{LSS}_X)$ has bounded scale measure we know there exists a natural number $u(f^{-1}(\mathcal{V}))$ such that $|f^{-1}(B_V)| \leq u(f^{-1}(\mathcal{V}))$. Further, since $f$ is a function we have that $|B_V|\leq |f^{-1}(B_V)|$. Finally, setting $u(\mathcal{V})= u(f^{-1}(\mathcal{V}))$ gives us that for any $f(\mathcal{U})$-net of $V \in \mathcal{V}$ has at most $u(\mathcal{V})$ elements and thus $(Y,\mathcal{LSS}_Y)$ has bounded scale measure at scale $f(\UU)$.\\
Conversely, let $(Y,\mathcal{LSS}_Y)$ be a space with bounded scale measure and let $f:X\rightarrow Y$ be a coarse equivalence. By theorem 2.2 there exists a scale $\mathcal{U}$ in $\mathcal{LSS}_Y$ such that for all scales $\mathcal{V}$ of $\mathcal{LSS}_Y$ there exists a natural number $k(\mathcal{V})$ so that for all $V \in \mathcal{V}$, $V\subseteq \bigcup\limits_{i=1}^{k(\mathcal{V})} U_i$ for a collection $U_i\in \mathcal{U}$. Since $f$ is a coarse equivalence, then $f^{-1}(\mathcal{U})$ is a uniformly bounded cover of $X$. Take any uniformly bounded cover $\mathcal{V}$ of $\mathcal{LSS}_X$. Then $f(\mathcal{V})$ is a cover of $\mathcal{LSS}_Y$. From the definition, there exists a collection of elements $U_i \in \mathcal{U}$ with $f(\mathcal{V})\subset\bigcup\limits_{i=1}^{k(f(\mathcal{V}))}U_i$. Thus $V\subset f^{-1}(f(V))\subset \bigcup\limits_{i=1}^{k(f(\mathcal{V}))}f^{-1}(U_i)$. Since $\mathcal{V}$ was chosen arbitrarily, $(X, \mathcal{LSS}_X)$ has bounded scale measure at scale $f^{-1}(\mathcal{U})$.
\end{proof}






We now present some examples of spaces with bounded scale measure:

\begin{Definition}
A group $G$ has a \textbf{proper left invariant metric} if the topology generated by the metric on the group is proper (i.e. the topological closure of a bounded set is compact) and additionally for any $g,h,k\in G$ $d(kg,kh)=d(g,h)$.
\end{Definition}

\begin{Remark}
It's shown in \cite{NowakYu} that any countable group has such a metric. Further, $\mathbb{R}^n$ has the euclidean metric as a proper left invariant metric where $\mathbb{R}^n$ is viewed as a group under addition.
\end{Remark}

\begin{prop}
Let $G$ be a group with a proper left invariant metric and let $\mathcal{LSS}$ be the large scale structure induced from the metric. Then $(G,\mathcal{LSS})$ has bounded scale measure for any open cover of balls of radius $R>0$. 
\end{prop}

\begin{proof}
Let $R>0$ and let $\UU$ be the cover by open $R$ balls. Let $\VV$ be any other scale in the large scale structure. Since our large scale structure is induced from a metric, there exists $S>0$ such that the scale consisting of $S$ balls coarsens $\VV$. We will call this scale $\WW$. With the topology induced by the metric in mind, let $W \in \WW$ be a specific element and consider the topological closure of $W$. This is a compact set due to the metric being proper and as $\UU$ is an open cover, there is a finite subcollection $U_1,...,U_k\in\UU$ so that $W\subseteq\bigcup\limits_{i=1}^{k}U_i$. Now, let $W'$ be any other element of $\WW$. Let $g$ and $h$ be the points of the centers of $W$ and $W'$ respectively. Then, one can obtain a cover of $W'$ by translating the centers of $U_1,...,U_k$ by the element $hg^{-1}$.
This implies that every $W'\in\WW$ can be covered by $k$ elements of $\UU$. Namely, the balls $hg^{-1}\cdot U_1,...,hg^{-1}\cdot U_k$. These are elements of $\UU$ due to left-invariance. Hence any element of $\VV$ can be covered by at most $k$ elements of $\UU$ and therefore $(G, \mathcal{LSS})$ has bounded scale measure for the scale $\UU$ of $R$ balls by theorem 2.2.
\end{proof}

From the proposition and the remark, we have the following result.

\begin{Corollary}
Let $\mathds{R}^n$ have the large scale structure induced by the Euclidean metric. Then $(\mathds{R}^n,\mathcal{LSS})$ has bounded scale measure at scale $\UU$ where $\UU$ is any cover of balls with radius $R>0$.
\end{Corollary}

A question we have is whether or not translation invariance is needed. This is to ask if any proper metric space has bounded scale measure at any cover of balls of radius $R>0$.

\begin{Question}
Let $M$ be a proper metric space and $\mathcal{LSS}$ be the large scale structure induced by the metric. Does $(M,\mathcal{LSS})$ have bounded scale measure for any cover of $R$ balls of radius larger than zero?
\end{Question}

The following is an application of theorem 2.6.

\begin{Example}
Let $\mathds{R}$ have the large scale structure induced from the metric of absolute value. This will follow because it's known that $\mathds{Z}$ with the absolute value metric is coarsely equivalent to $\mathds{R}$. Since $\mathds{Z}$ with the large scale structure from the absolute value metric has bounded geometry (i.e. has bounded scale measure at the scale of the cover of one point sets), then $\mathds{R}$ with the large scale structure induced from the absolute value metric also has bounded scale measure. It's important to note here that while $\mathds{R}$ does not have bounded geometry, it has bounded scale measure by the above corollary. As shown earlier in the paper, bounded geometry is not in general preserved by coarse equivalences. However, bounded scale measure is preserved by coarse equivalences.
\end{Example}

Here is another example of a large scale space having bounded scale measure.

\begin{Example}
Let $\mathds{R}$ have the large scale structure $\mathcal{LSS}$ induced by the discrete metric. Then $(\mathds{R},\mathcal{LSS})$ has bounded scale measure for any cover by open balls of radius $R\geq 1$. Note that any ball of radius greater than or equal to 1 is all of $\mathds{R}$. However, if $\UU\in\mathcal{LSS}$ is a cover by balls of radius less than 1, then $\UU$ is a cover of single points. By choosing $\mathcal{V}$ to be the cover by balls of radius 1, then each element of $\VV$ is all of $\mathds{R}$ hence no element of $\VV$ can be covered by a finite amount of elements of $\UU$.
\end{Example}

We present an example of a space that does not have bounded scale measure at any scale.

\begin{Example}
Consider the space $X=\mathds{Z}\times\mathds{Z}^2\times\mathds{Z}^3\times ...$. An element of $X$ will be denoted as $\vec{x}=(x_1, (x_{2, 1}, x_{2, 2}), (x_{3, 1}, x_{3, 2}, x_{3, 3}),...)$. Consider the metric $d_i$ of the $\mathbb{Z}^i$ component defined as $d_i(\vec{x}_i,\vec{y}_i)=\max\limits_{1\leq k\leq i} |x_{i, k}-y_{i, k}|$. Note here that balls with such a metric look like squares of appropriate dimension. One can then construct an $\infty$-metric $d$ on $X$ defined as $$d(\vec{x},\vec{y})=\sum\limits_{i=1}^{\infty} d_i(\vec{x}_i,\vec{y}_i)$$  Consider the large scale structure induced by this $\infty$-metric. Since any metrizable large scale structure is countably generated by balls of radius $0,1,2,..$ we only need to show that any cover by radius $n$ balls fails to cover all elements of some scale $\VV$ with a finite union of $n$ balls and will thus fail to have bounded scale measure by theorem 2.2.

Let $\UU$ be a cover by $n$ balls where $n$ is a non-negative integer. Let $\VV$ be the cover of $n+1$ balls. We will show that for all natural numbers $k$ there does not exist a collection $U_i \in\UU$ with $V\subseteq \bigcup\limits_{i=1}^{k}U_i$ for some $V \in \VV$.  In the case of $n=0$, this happens because a 1-ball in each component contains an increasingly large number of points which implies that a $1$-ball in $X$ contains infinitely many points $\vec{x}$.

In the case of $n>0$ this is because each element of $\VV$ has at most $(n+1)^i$ elements in its $\mathbb{Z}^i$ component, each element of $\UU$ has at most $n^i$ elements in its $\mathbb{Z}^i$ component, and the number $(n+1)^i-n^i$ is unbounded as $i$ gets large. Thus the cover $\VV$ can't be covered by a finite number of elements of $\UU$.
\end{Example}

\section{Generalizations of Property A}
Property A was defined by Guoliang Yu \cite{yu2000} to approach the Baum-Connes conjecture. It can be viewed as a weaker version of amenability and is a sufficient condition for many properties such as coarse embeddability into a Hilbert Space. In its original context, property A was defined for a uniformly discrete metric space in the following way. 
\begin{Definition}
Let $X$ be a uniformly discrete metric space. $X$ has \textbf{property A} if for every $\epsilon >0$ and $R>0$ there exists a collection of finite subsets $\{A_x\}_{x\in X}, A_x \subseteq X \times \mathds{N}$ for every $x\in X$, and a constant $S>0$ such that
\begin{itemize}
\item $\frac{|A_x \triangle A_y|}{|A_x \cap A_y|}\leq \epsilon$ when $d(x,y)\leq R$
\item $A_x \subseteq B(x,S)\times \mathds{N}$
\end{itemize}
\end{Definition}
These conditions make the sets $A_x$ and $A_y$ be almost equal if $d(x,y) \leq R$ and disjoint if $d(x,y)\geq 2S$. From there, Hiroki Sako defined property A in \cite{Sako} for uniformly locally finite coarse spaces. Our definition is inspired by \cite{Sako}:

\begin{Definition} 
Let $X$ be a set and let $\mathcal{LSS}$ be a Large Scale structure on $X$ with bounded geometry. ($X, \mathcal{LSS}$) is said to have \textbf{Property A} if for every $\epsilon >0$ and every uniformly bounded family $\mathcal{U} \in \mathcal{LSS}$, there exists $\mathcal{V} \in \mathcal{LSS}$ and a family of finite subsets $\{A_x^{\mathcal{LSS}}\}$ of $X\times \mathds{N}$ such that
\begin{itemize}
\item $A_x^{\mathcal{LSS}} \subset st(x,\mathcal{V})\times \mathds{N}$
\item $(x,1) \in A_x^{\mathcal{LSS}}$
\item $|A_x^{\mathcal{LSS}} \Delta A_y^{\mathcal{LSS}}| < \epsilon |A_x^{\mathcal{SLS}} \cap A_y^{\mathcal{LSS}}|$ if $y \in st(x,\mathcal{U})$.
\end{itemize}
\end{Definition}

It can be shown that our definition of property A for large scale spaces with bounded geometry is ``equivalent'' in some sense to Sako's definition of property A for uniformly locally finite coarse spaces. We will state the theorem here. The interested reader can read a refresher on coarse spaces and the proof of the following theorem in the appendix at the end of this paper:

\begin{Theorem}
Let $(X,\mathcal{LSS})$ be a large scale space with bounded geometry. If a uniformly locally finite coarse space $(X,\mathcal{C})$ has property A, then the induced large scale structure ($X, \mathcal{LSS}$) has Property A. Further, if ($X, \mathcal{LSS}$) has property A, then the induced coarse structure $(X,\mathcal{C})$ also has property A.
\end{Theorem}

Because of this theorem, we will drop the superscripts from the collections $\{A_x\}_{x\in X}$. We now show that our definition of property A is a coarse invariant.

\begin{Theorem}
Let $(X,\mathcal{LSS}_X)$ and $(Y,\mathcal{LSS}_Y)$ be large scale spaces with bounded geometry. Let $f:X \rightarrow Y$ be a coarse equivalence. If $(Y, \mathcal{LSS}_Y)$ has property A, then $(X, \mathcal{LSS}_X)$ also has property A.
\end{Theorem}
\begin{proof}
Let $\mathcal{U}$ be a uniformly bounded family of $\mathcal{LSS}_X$ and let $\epsilon >0$ be given. Let $N=sup_{y\in Y}|f^{-1}(y)|$. Let $\mathcal{W}$ be the uniformly bounded family of $\mathcal{LSS}_Y$ and $\{B_x\}$ the finite subsets of $Y \times \mathds{N}$ from the definition of property A for $f(\mathcal{U})$ and $\frac{\epsilon}{N}$. Finally, define $\mathcal{V}=f^{-1}(\mathcal{W})$ and for each $x\in X$ define $$A_x =\{(z,n)\in X \times \mathds{N} | (f(z),n)\in B_{f(x)}\}.$$
By definition of $\VV$, it holds trivially that $A_x \subset st(x,\mathcal{V}) \times \mathds{N}$ and that $(x,1)\in A_x$. \\
Finally, we check that if $y\in st(x,\mathcal{U})$ then $|A_x \triangle A_y| < \epsilon |A_x \cap A_y|$. If $y\in st(x,\mathcal{U})$, then $f(y) \in st(f(x),f(\mathcal{U}))$ and 
$$|B_{f(x)}\triangle B_{f(y)}| < \frac{\epsilon}{N} |B_{f(x)}\cap B_{f(y)}|.$$
From the construction of $A_x$'s and bounded geometry it follows that $|A_x\triangle A_y|\leq N \cdot |B_{f(x)}\triangle B_{f(y)}|$ and as $f$ is a function
$|A_x \cap A_y| \geq |B_{f(x)}\cap B_{f(y)}|$. Putting this all together we get
$$|A_x\triangle A_y|\leq N \cdot |B_{f(x)}\triangle B_{f(y)}| \leq N\cdot \frac{\epsilon}{N} |B_{f(x)}\cap B_{f(y)}| \leq \epsilon |A_x \cap A_y|,$$
and therefore $(X,\mathcal{LSS}_X)$ has property A.
\end{proof}

The large scale definition of property A allows us to approach relationships with other coarse properties in a more topological way. As such, we will use this definition to show that finite asymptotic dimension implies the large scale definition of property A.\\
Asymptotic Dimension can be thought of as a large scale geometric version of the covering dimension. One of the first useful definitions involved the concept of $R$-multiplicity of a scale on a metric space.

\begin{Definition}
For a scale $\mathcal{U}$, the \textbf{R-multiplicity} of $\mathcal{U}$ on a metric space is the smallest positive integer $n$ such that for every $x\in X$ the ball $B(x,R)$ intersects at most $n$ elements of $\mathcal{U}$.
\end{Definition}

This in turn inspired the following definition from \cite{Dydak}:
\begin{Definition}
Suppose that $X$ is a metric space. The \textbf{asymptotic dimension} of $X$ is the smallest positive integer $n$ such that for every $R>0$ there exists a uniformly bounded cover $\mathcal{U}$ with $R$-multiplicity $n+1$. It is denoted by $asdim(X) =n$.
\end{Definition}
From there, asymptotic dimension was generalized for large scale structures in \cite{Dydak}.

\begin{Definition}
Let $(X,\mathcal{LSS})$ be a large scale space with large scale structure $\mathcal{LSS}$. We say $asdim(X,\mathcal{LSS}) \leq n$ if for every uniformly bounded cover $\mathcal{U}$ in $\mathcal{LSS}$ there exists a uniformly bounded cover $\mathcal{V}$ such that $\mathcal{V}$ is a coarsening of $\mathcal{U}$ with multiplicity at most $n+1$ (i.e. each point $x\in X$ is contained in at most $n+1$ elements of $\mathcal{V}$).
\end{Definition}
 Since the trivial cover of $X$ is uniformly bounded, this definition gives us that there exists a uniformly bounded cover $\mathcal{B}$ such that each $x \in X$ is contained in at most $n+1$ elements of $\mathcal{B}$. It can be shown that $asdim(X, \mathcal{LSS})\leq n$ if $\mathcal{LSS}$ can be generated by a uniformly bounded family $\mathcal{B}$ such that the multiplicity of $\mathcal{B}$ is at most $n+1$ \cite{Dydak}.
As for the metric case, we will show that finite asymptotic dimension for large scale spaces implies the general large scale definition of property A. First we will introduce some notation and a lemma.\\

\begin{Notation}
Given a uniformly bounded cover $\mathcal{U}$ for all positive integers $n$, define $st^{n}(\mathcal{U})$ in the following way:\\
$st^{0}(\mathcal{U}) = \mathcal{U}$ \\
$st^{n}(\mathcal{U})=st(\mathcal{U},st^{n-1}(\mathcal{U}))$
\end{Notation}

\begin{Lemma}
If $y\in st(x,\mathcal{U})$, then $st(y,st^{n-1}(\mathcal{U})) \subseteq st(x,st^{n}(\mathcal{U})).$
\end{Lemma}
\begin{proof}
Let $z\in st(y, st^{n-1}(\mathcal{U}))$ be arbitrarily chosen. Then there exists $U\in \mathcal{U}$ such that $z,y \in st(U, st^{n-2}(\mathcal{U}))$. Since $y\in st(x, \mathcal{U})$ there exists $U'\in \mathcal{U}$ with $x,y \in U'$. Therefore we have 
$$U' \cap st(U, st^{n-2}(\mathcal{U})) \neq \emptyset$$
and thus by the definition of stars and the above we have $z\in st(U, st^{n-2}(\mathcal{U})) \subset st(U', st(\mathcal{U}, st^{n-2}(\mathcal{U}))).$ 
Putting this all together we then have 
$$x,z \in st(U',st(\mathcal{U},st^{n-2}(\mathcal{U})) = st(U', st^{n-1}(\mathcal{U})) \subseteq st^{n}(\mathcal{U})$$
which shows $z \in st(x, st^{n}(\mathcal{U}))$ and therefore $st(y,st^{n-1}(\mathcal{U})) \subseteq st(x,st^{n}(\mathcal{U})).$
\end{proof}
\begin{Theorem}
Let $(X, \mathcal{LSS})$ be a large scale space. If $asdim (X,\mathcal{LSS})$ is finite, then $(X, \mathcal{LSS})$ has property A.
\end{Theorem}
\begin{proof}
Let $asdim(X, \mathcal{LSS})=k$. Let $\epsilon > 0$ be given and $\mathcal{U}$ an arbitrary uniformly bounded family. Then for any $n\in \mathbb{N}$ we have that $st(\{x\}_{x\in X}, st^{n}(\mathcal{U}))$ is a uniformly bounded family. From the definition of finite asymptotic dimension, let $\mathcal{V}$ be a uniformly bounded family such that $st(\{x\}_{x\in X}, st^{n}(\mathcal{U})) \prec \mathcal{V}$ with multiplicity at most $k+1$. For each $V\in \mathcal{V}$ arbitrarily pick $z\in V$ and denote it $z_V$. Define a uniformly bounded family $\mathcal{W}=st(st(\{x\}_{x\in X}, st^{n}(\mathcal{U})), \mathcal{V})$. For each $x\in X$ define $A_x$ in the following way:
$$A_x:=\{(x,1)\} \cup \{(z_V,m) \mid st(x, st^{m}(\mathcal{U}))\cap V \neq \emptyset, V \nsubset st(x,st^{m}(\mathcal{U})) , 1\leq m\leq n\}.$$
Then each $A_x$ is finite due to the multiplicity of $\mathcal{V}$ against $st(x,st^{n}(\mathcal{U}))$ is at most $k+1$ for all $x\in X$. It also follows that $A_x \subseteq st(x, \mathcal{W})\times \mathds{N}$ and $(x,1)\in A_x$. All that remains is to check that the inequality in the definition of property A is satisfied.

Let $y\in st(x, \mathcal{U})$.
Since $st(\{x\}_{x\in X}, st^{n}(\mathcal{U})) \prec \mathcal{V}$, there exists $V\in \mathcal{V}$ with $st(x,st^{n}(\mathcal{U}))\subseteq V$. From the construction of $A_x$ for this fixed $V\in \mathcal{V}$ we have $(z_V,m)\in A_x$ for $1\leq m\leq n$. Additionally, we have $y\in V$ and from our lemma we have $(z_V, m)\in A_y$ for $1\leq m \leq n-1.$ Therefore, $|A_x \cap A_y| \geq n-1$. \\
On the other hand, consider $(z_V,l)\in A_x \triangle A_y$.
Without loss of generality suppose $(z_V,l)\in A_x \setminus A_y$. Therefore we know that $V\cap st(x, st^{l}(\mathcal{U}))\neq \emptyset$ and $V \not\subset st(x,st^{l}(\mathcal{U}))$. On the other hand, we know that either $V\cap st(y, st^{l}(\mathcal{U}))= \emptyset$ or $V \subset st(y,st^{l}(\mathcal{U}))$. \\
First suppose that $V\cap st(y, st^{l}(\mathcal{U}))= \emptyset$. From the previous lemma we have that $st(x, st^{l}(\mathcal{U})) \subset st(y, st^{l+1}(\mathcal{U}))$, so $V\cap st(y, st^{l+1}(\mathcal{U})) \neq \emptyset$. On the other hand, suppose $V \subset st(y,st^{l}(\mathcal{U}))$. Again, by the above lemma we have $V \subset st(y,st^{l}(\mathcal{U}))\subseteq st(x, st^{l+1}(\mathcal{U}))$. Therefore any $(z_V,l)$ can only be contained in $A_x \setminus A_y$ for at most 2 distinct values of $l$. Each $A_x$ can only have at most $k+1$ distinct $z_V$'s so $|A_x \triangle A_y| \leq 2(2(k+1)+1)$. Putting both inequalities together we get that
$$\frac{|A_x \triangle A_y|}{|A_x \cap A_y|} \leq \frac{4k+6}{n-1}.$$
The numerator is independent of our choice of $n$, so we may choose an $n$ great enough so that $\frac{4k+6}{n-1} < \epsilon$ to satisfy the inequality. Therefore the collection $\{A_x\}_{x\in X}$ satisfies the requirements and $(X, \mathcal{LSS})$ has property A.
\end{proof}

At this juncture, we seek to construct a definition of property A that does not depend on the large scale structure having bounded geometry. We finally consider a generalization of property A that occurs in spaces with bounded scale measure. 

If $(X,\mathcal{LSS})$ has bounded scale measure at scale $\mathcal{U}$, then it's useful to imagine the elements of $\mathcal{U}$ as ``points''. In our previous definition of property A, we needed our points in a specific $A_x$ to be ``near'' one another ($A_x \subset st(x,\mathcal{V})\times \mathds{N}$). To accomplish this property for our ``points'' $U\in\mathcal{U}$, we will make use of a set called a horizon:

\begin{Definition}
The \textbf{horizon} of a set against a scale is defined in the following way:
$$hor(A,\UU)=\left\{U_i ~|~ A\cap U_i \neq \emptyset, U_i \in \UU \right\}$$
\end{Definition}

We now introduce the definition of property A for a large scale space with bounded scale measure:

\begin{Definition}
Let $(X, \mathcal{LSS}_X)$ be a large scale space with bounded scale measure. We say that $(X,\mathcal{LSS})$ has \textbf{property A at scale} $\mathcal{U}$ if for all $\epsilon>0$ and for all $\mathcal{V} \in \mathcal{LSS}_X$ that is not the trivial cover, there exists a collection of finite subsets $\left\{A_U\right\}_{U\in \mathcal{U}}$ with $A_U \subseteq \mathcal{U} \times \mathds{N}$ for all $U\in \mathcal{U}$ and $\mathcal{W} \in \mathcal{LSS}_X$ such that:

\begin{description}

\item{1.)} For all $U \in \mathcal{U}, U\times\left\{1 \right\} \in A_U$

\item{2.)} $A_U \subseteq hor(st(U,\mathcal{W}), \mathcal{U})$

\item{3.)} $\frac{|A_{U_1} \Delta A_{U_2}|}{|A_{U_1} \cap A_{U_2}|}< \epsilon$ whenever $hor(U_1,\mathcal{V})\cap hor(U_2,\mathcal{V}) \neq \emptyset$
\end{description}
\end{Definition}
\begin{Remark}
We note here that while each element of $\left\{ A_U\right\}_{U\in\mathcal{U}}$ is a collection of finite subsets of $\mathcal{U}$, $\left\{ A_U\right\}_{U\in\mathcal{U}}$ itself may not contain finitely many points of $X$. \\

We also note here that if one has property A at scale $\mathcal{E}$ where $\mathcal{E}$ is the trivial cover, then we have the definition of property A that was discussed earlier in the paper. \\

Lastly, it's useful to think of having property A at scale $\mathcal{U}$ as having property A but consider the elements of $\UU$ as ''points" as opposed to elements of $X$. In doing this, we can contain complexity within elements of $\mathcal{U}$ and see if the the space has property A-like behavior at a larger scale.
\end{Remark}

We now show having property A at a scale is a coarse invariant.

\begin{Theorem}
Let $(X,\mathcal{LSS}_X)$ and $(Y,\mathcal{LSS}_Y)$ be large scale spaces with bounded scale measure. Let $f:X\rightarrow Y$ be a large scale equivalence. If $(Y, \mathcal {LSS}_Y)$ has property A at a scale, then $(X, \mathcal{LSS}_X)$ also has property A at a scale.
\end{Theorem}
\begin{proof}
Let $(X,\mathcal{LSS}_X)$ have bounded scale measure at scale $\UU_X$. Let $(Y,\mathcal{LSS}_Y)$ have bounded scale measure at scale $\mathcal{U}_Y$ and let $\mathcal{LSS}_Y$ have property A at scale $\mathcal{U}_Y$. By theorem 2.2, let $n$ be so that any element of the scale $f^{-1}(\UU_Y)$ is covered by at most $n$ elements of $\UU_X$; likewise, let $m$ be so that any element of the scale $f(\UU_X)$ can be covered by at most $m$ elements of $\UU_Y$. Let $\mathcal{V}_X$ be any scale in $\mathcal{LSS}_X$ and let $\epsilon > 0$ be given. 
Since $f$ is a coarse equivalence both $f(\UU_X)$ and $f(\VV_X)$ are scales in $\mathcal{LSS}_Y$. Therefore $st(f(\VV_X), f(\UU_X))$ is a uniformly bounded family in $\mathcal{LSS}_Y$ and we will set $st(f(\VV_X), f(\UU_X)) = \VV_Y$. 
Then as $(Y,\mathcal{LSS}_Y)$ has property $A$ at scale $\UU_Y$, for $\VV_Y$ and $\frac{\epsilon}{2m^2(n+1)}$ there exists a scale $\WW_Y$ of $\mathcal{LSS}_Y$ and a collection of finite subsets $\left\{B_{U_Y}\right\}$ of $\UU_Y \times \mathds{N}$ such that:
\begin{description}
\item{1.)} For all $U_Y \in \UU_Y, U_Y\times\left\{1\right\} \in B_{U_Y}$
\item{2.)} $B_{U_Y} \subseteq hor(st(U_Y,\WW_Y), \mathcal{U}_Y)$
\item{3.)} $\frac{|B_{U_{Y_1}} \Delta B_{U_{Y_2}}|}{|B_{U_{Y_1}} \cap B_{U_{Y_2}}|}< \frac{\epsilon}{2m^2(n+1)}$ whenever $hor(U_{Y_1},\mathcal{V}_Y)\cap hor(U_{Y_2},\mathcal{V}_Y) \neq \emptyset$
\end{description}
Since $f(\UU_X)$ is a uniformly bounded scale, it must be that for all $U_X \in \UU_X$ there exists $m$ elements of $\UU_Y$ such that $f(U_X) \subseteq \bigcup\limits_{i=1}^{m} U_{Y_i}$. Each $U_{Y_i}$ has a corresponding $B_{U_{Y_i}}$. 
Additionally, $f^{-1}(\mathcal{U}_Y)$ is a scale in $X$ and therefore each element of this scale can be covered by $n$ elements of $\UU_X$ from the definition of bounded scale measure. Thus, for each $f^{-1}({U_{Y_i}})$ we choose $n$ elements of $\UU_X$ that cover it, i.e. $f^{-1}(U_{Y_i})\subseteq \bigcup\limits_{i=1}^{n} U_{X_{i}}$. 
When constructing $A_{U_{X}}$ we will only be using the $m$ elements from $\UU_Y$ that we previously chose such that $f(U_X) \subseteq \bigcup\limits_{i=1}^{m} U_{Y_i}$. Similarly, each $U_{X_i}$ is from the chosen $n$ elements of $\UU_X$ where $f^{-1}(U_{Y_i})\subseteq \bigcup\limits_{i=1}^{n} U_{X_{i}}$ for some $1\leq i\leq m$. 
For $U_X \in \UU_X$ define $$A_{U_X} = \left\{ (U_X, 1) \right\} \bigcup \left\{ (U_{X_i}, n) \in \UU \times \mathds{N}~|~hor(f(U_{X_i}),\UU_Y) \times \{n\} \cap \bigcup\limits^{m}_{j=1} B_{U_{Y_j}} \neq \emptyset  \right\}$$
where $U_{X_i}$ and $B_{U_{Y_j}}$ are from the chosen $n$  and $m$ elements respectively.
As well, define $\WW_X = f^{-1}(st(\WW_Y, \UU_Y))$.
We must now show that our chosen $A_{U_X}$ and $\WW_X$ satisfy the property A at scale $\UU_X$. By construction, each $A_{U_X}$ contains the element $(U_X , 1)$.

Next, we will now show that any $A_{U_X}$ has the property that $A_{U_X}\subseteq hor(st(U_X,\WW_X),\UU_X)$. 
Indeed, let $(U_{X_i},l)\in A_{U_X}$ for some $l\in \mathbb{N}$. By the choice of the $m$ elements of $\UU_Y$ for the definition of $A_{U_X}$, we have that $f(U_{X})\cap U_{Y_k}\neq\emptyset$ for all $U_{Y_k}$ for $1\leq k \leq m$.
Now, since $(U_{X_i},l)\in A_{U_X}$ we also have that $hor(f(U_{X_i}),\UU_Y)\times\{k\}\cap B_{U_{Y_j}}\neq\emptyset$ for some $k\in \mathbb{N}$ and some $j$ with $1\leq j \leq m$. Using the previous statement, we may define $(U_{Y_1},l)\in B_{U_{Y_j}}$ with the property that $f(U_{X_i}) \cap U_{Y_1} \neq\emptyset$. 
Because any $B_{U_{Y_j}}\subseteq hor(st(U_{Y_j},\WW_Y),\UU_Y)$, we have that $U_{Y_1}\in hor(st(U_{Y_j},\WW_Y),\UU_Y)$. Thus there exists $W_Y \in \WW_Y$ with the property that $W_Y \cap U_{Y_j} \neq \emptyset$ and $W_Y \cap U_{Y_1} \neq \emptyset$ from the definition of the star and the horizon.
Observe that the set $U_{Y_j}\cup W_Y\cup U_{Y_1}$ is contained in an element of $st(\WW_Y,\UU_Y)$. We will call this element $C$. Then taking the inverse image one has $f^{-1}(C)$ is an element of $f^{-1}(st(\WW_Y,\UU_Y))$ which is defined to be an element of $\WW_X$.
Furthermore, $U_{X_i}\cap f^{-1}(C)\neq\emptyset$
because $U_{X_i} \cap  f^{-1}(U_{Y_1}) \neq \emptyset$ and
$f^{-1}(U_{Y_1})\subseteq f^{-1}(U_{Y_j}\cup W_Y\cup U_{Y_1})\subseteq f^{-1}(C)$. 
Similarly, since $U_X \cap f^{-1}(U_{Y_j}) \neq \emptyset$ and $f^{-1}(U_{Y_j}) \subseteq f^{-1}(U_{Y_j}\cup W_Y\cup U_{Y_1})\subseteq f^{-1}(C)$ 
we have that $f^{-1}(C)\cap U_X \neq \emptyset$. Therefore, we have that $U_{X_i}\in hor(st(U_X,\WW_X),\UU_X)$ as desired.

Finally we must show that given any $U_{X_1}, U_{X_2} \in \UU_X$ with \\
$hor(U_{X_1}, \VV_X) \cap hor (U_{X_2}, \VV_X) \neq \emptyset$, then $\frac{|A_{U_{X_1}} \Delta A_{U_{X_2}}|}{|A_{U_{X_1}} \cap A_{U_{X_2}}|}< \epsilon$.
To begin, recall that
$$A_{U_{X}} = \left\{(U_X, 1) \right\} \bigcup \left\{(U_{X_i}, n) \in \UU \times \mathds{N} ~|~ hor(f(U_{X_i}),\UU_Y) \times \left\{n \right\} \cap \bigcup^{m}_{j=1}B_{U_{Y_j}} \neq \emptyset  \right\}$$
for a chosen collection of $m$ elements $U_{Y_j}$ from $\UU_Y$.
Additionally, each $U_{Y_j}$ contains at most $n$ chosen $f(U_{X_i})$'s and the original $U_{X}$. This combined with the fact that both $f(U_{X_1})$ and $f(U_{X_2})$ may intersect $m$ unique elements taken from the $B_{U_{Y_j}}$ each, 
we have 
$$|A_{U_{X_1}} \Delta A_{U_{X_2}}| \leq 2m(n+1)\cdot max\left\{ |B_{U_{Y_k}} \Delta B_{U_{Y_l}}|~ | 1\leq k \leq m, 1\leq l\leq m \right \}.$$ 
However, since we have $hor(U_{X_1}, \VV_X) \cap hor (U_{X_2}, \VV_X) \neq \emptyset$ then for any of the counted $B_{U_{Y_k}}$ or $B_{U_{Y_l}}$ for $U_{X_1}$ and $U_{X_2}$ respectively 
it must be that \\$hor(U_{Y_k}, \VV_Y) \cap hor(U_{Y_l}, \VV_Y) \neq \emptyset$ as $\VV_Y = st(f(\VV_X), f(\UU_X))$. Therefore, 
$$|B_{U_{Y_k}} \Delta B_{U_{Y_l}}| < \frac{\epsilon}{2m(n+1)}\cdot |B_{U_{Y_k}} \cap B_{U_{Y_l}}|$$ for all $1\leq k\leq m, 1\leq l\leq m$.
We will finish by showing that $|A_{U_{X_1}} \cap A_{U_{X_2}}| \geq m\cdot |B_{U_{Y_k}} \cap B_{U_{Y_l}}|$ for all counted $B_{U_{Y_k}}$ and $B_{U_{Y_l}}$. 
First of all, let $U_{Y_*}\times \{z\}\in B_{U_{Y_l}} \cap B_{U_{Y_k}}$. We know there exists $n$ elements of $\UU_X$ that cover $f^{-1}(U_{Y_*})$. Thus there exists at least one element $U_{X_*} \in \UU_X$ with $f(U_{X_*})\cap U_{Y_*}\neq \emptyset.$ Therefore $U_{X_*}\times \{z\} \in A_{U_{X_1}} \cap A_{U_{X_2}}$. Additionally, we know that there exist $m$ elements of $\UU_Y$ that cover $f(U_{X_*})$ by construction. Thus, there are at most $m$ elements of $B_{U_{Y_l}} \cap B_{U_{Y_k}}$ that can correspond with $U_{X_*}\in A_{U_{X_1}} \cap A_{U_{X_2}}$, and $|A_{U_{X_1}} \cap A_{U_{X_2}}| \geq m\cdot |B_{U_{Y_k}} \cap B_{U_{Y_l}}|$ as we wanted.
Finally, we choose a specific $k$ and $l$ so that $|B_{U_k} \Delta B_{U_l}| = max\left\{ |B_{U_{Y_k}} \Delta B_{U_{Y_l}}|~ | 1\leq k \leq m, 1\leq l\leq m \right \}.$
Combining this altogether, we get 
$$\frac{|A_{U_{X_1}} \Delta A_{U_{X_2}}|}{|A_{U_{X_1}} \cap A_{U_{X_2}}|} \leq \frac{2m^2(n+1)|B_{U_{Y_l}} \Delta B_{U_{Y_k}}|}{|B_{U_{Y_l}} \cap B_{U_{Y_k}}|}<  \frac{2m^2(n+1)\cdot \epsilon}{2m^2(n+1)} = \epsilon$$
and $(X, \mathcal{LSS}_X)$ has property A at scale $\UU_X$.
\end{proof}

Using the previous proof and that bounded scale measure is a coarse invariant, one has the immediate corollary:

\begin{Corollary}
Let $(X,\mathcal{LSS}_X)$ and $(Y,\mathcal{LSS}_Y)$ be large scale structures with $f:X\to Y$ a coarse equivalence. If $(Y,\mathcal{LSS}_Y)$ has bounded scale measure and property A at a scale, then $(X,\mathcal{LSS}_X)$ has bounded scale measure and property A at a scale. 
\end{Corollary}

In closing, we would like to give a parting thought regarding possible connections between bounded scale measure and coarsely doubling spaces along with weakly paracompact spaces shown in \cite{cavp}. Many of these ideas are similar or even equivalent in metric spaces. One wonders how they may compare in a more general setting. Additionally, it may be worthwhile to consider connections between property A at a scale and finite asymptotic dimension.

\section{Appendix}

We begin the appendix with a proof that a large scale continuous function $f$ is a coarse equivalence if and only if $f$ is a coarse embedding and is coarsely surjective:

\begin{Theorem}
A large scale continuous map $f: X \rightarrow Y$ is a coarse equivalence if and only if $f$ is a coarse embedding and is coarsely surjective.
\end{Theorem}
\begin{proof}
Assume that $f: X \rightarrow Y$ is a coarse equivalence.
First we will show that $f$ is a coarse embedding. Let $\mathcal{U}$ be a uniformly bounded cover of $Y$ and let $g:Y\rightarrow X$ be the function such that $g\circ f:X\rightarrow X$ is close to the identity $id_X$. Then there exists a uniformly bounded family $\mathcal{V}$ of $X$ such that for every $x\in X$ there exists a $V\in \mathcal{V}$ with $x\in V$ and $g\circ f(x)\in V$. Let $\mathcal{U}$ be a uniformly bounded family in $Y$. Since $g$ is large scale continuous, then $g(\mathcal{U})=\{g(U):U\in \mathcal{U}\}$ is a uniformly bounded family in $X$. We claim the uniformly bounded family $\mathcal{W}=st(g(\mathcal{U}),\mathcal{V})$ coarsens $f^{-1}(\mathcal{U})$. \\
Let $U\in \mathcal{U}$. Consider any $x\in X$ such that $f(x)=y \in U$. From the definition of coarse equivalence, there exists a $V_y\in \mathcal{V}$ such that $\{g\circ f(x), x\}\subseteq V_y$. Therefore, $\{g(y),x\}\subseteq V_y$ for all $x\in X$ with $f(x)=y$, or $\{g(y),f^{-1}(y)\}\subseteq V_y$. \\
Now, for $U\in \mathcal{U}$, $f^{-1}(U)= \bigcup\limits_{y\in U} f^{-1}(y) \subseteq \bigcup\limits_{y\in U} V_y$ for $V_y$ defined above. However, for $y\in U$ we know $\{g(y),x\} \in V_y$ for all $x\in X$ with $f(x)=y$, so $V_y\cap g(U) \neq \emptyset$. Thus for all $y\in U$ $V_y \subseteq st(g(U),\mathcal{V})$. Therefore $f^{-1}(U) \subseteq st(g(U),\mathcal{V})$, and from the definition of a large scale structure we have $f^{-1}(\mathcal{U})$ is a uniformly bounded family hence $f$ is a coarse embedding. \\
Next we will show that $f$ is coarsely surjective. From the definition of coarse equivalence, there exists a uniformly bounded cover $\mathcal{U}$ of $Y$ such that for every $y \in Y$ there exists a $U \in \mathcal{U}$ with $\{f\circ g(y),y\}\subseteq U$. Let $y \in Y$ be arbitrary. Then there exists a $U \in \mathcal{U}$ with $\{f\circ g(y),y\}\subseteq U$. Therefore $y\in st(f\circ g(Y), \mathcal{U})$. This implies $Y \subseteq st(f\circ g(Y), \mathcal{U})$ and since $f\circ g(Y)\subseteq f(X)$ we have that $Y \subseteq st(f\circ g(X), \mathcal{U})$, and f is coarsely surjective.\\
Conversely, suppose $f$ is a large scale continuous map that is a coarse embedding and is coarsely surjective. First we will construct a large scale continuous function $g:Y\rightarrow X$. From coarse surjectivity, we know there exists a uniformly bounded family $\mathcal{U}$ in $Y$ such that $Y\subseteq st(f(X),\mathcal{U})$. Thus, for every $y \in Y$ there exists a $U_y \in \mathcal{U}$ with $U_y \cap f(X) \neq \emptyset$. Choose an $x_y\in X$ such that $f(x_y)\in U_y$. Define $g:Y\rightarrow X$ with $g(y)=x_y$. \\
We first show that $f\circ g$ is close to $id_Y$. Let $y\in Y$. From the construction of $g$, there exists $U_y \in \mathcal{U}$ with $y\in U_y$, and $f\circ g(y)\in U_y$. Thus $\{f\circ g (y), y\}\subseteq U_y$ and thus $f\circ g$ is close to $id_Y$. 
Now we will show that $g \circ f$ is close to $id_X$. Let $y\in Y$. By coarse embeddability we know $f^{-1}(\mathcal{U})$ is a uniformly bounded family in $X$. Choose $x\in X$. We know from coarse surjectivity that there exists $U_{f(x)} \in \mathcal{U}$ with both $f(x) \in U_{f(x)}$ and $f\circ g\circ f(x) \in U_{f(x)}$ from the construction of the function $g$. Therefore $\{g\circ f(x), x\} \subseteq f^{-1}(U_{f(x)})=\{z\in X | f(z)\in U_{f(x)}\}$, and $g \circ f$ is close to $id_X$. \\
Finally we will show that $g$ is large scale continuous. Let $\mathcal{U}$ be a uniformly bounded family of $Y$. Since $f\circ g$ is close to $id_Y$, then there exists a uniformly bounded family $\mathcal{V}$ of $Y$ such that for every $y\in Y$ there exists $V\in \mathcal{V}$ with $\{f\circ g(y), y\} \subseteq V$. Let $U \in \mathcal{U}$. For every $y\in U$, there exists $V \in \mathcal{V}$ such that $\{f\circ g(y), y\} \subseteq V$. Therefore, for every $U\in \mathcal{U}$, $f\circ g(U) \subset st(U,\mathcal{V})$, and we have that $f\circ g(\mathcal{U})$ is uniformly bounded in $Y$. Finally, since $f$ is a coarse embedding $f^{-1}\circ f \circ g(\mathcal{U})$ is a uniformly bounded family and for all $U \in \mathcal{U}$, $g(U) \subseteq f^{-1} \circ f \circ g(U)$. This means that $g(\mathcal{U})$ is uniformly bounded in $X$ from the definition of large scale structure. 
\end{proof}

We will now talk about coarse structures and provide a proof for theorem 3.1. Coarse structures were introduced by Higson and Roe for use in index theory and signature theory. Coarse structures were to give an approach for the Novikov and Coarse Baum-Connes conjectures. Much like scales and families of subsets of a set $X$, a metric space $X$ gives a natural example of what is to follow.

\begin{Definition}
	Let $X$ be a set and consider the set $X\times X$:
	\begin{enumerate}
		\item The \textbf{diagonal} of $X$ is denoted by $\Delta$ and is defined as $\Delta =\{(x,x)|~x\in X\}$.
		\item Let $U\subseteq X\times X$. Define the \textbf{inverse} of $U$, denoted by $U^{-1}$, to be $U^{-1}=\{(y,x)|~(x,y)\in U\}$.
		\item Let $U,V\subseteq X\times X$ and define the \textbf{product} of $U$ and $V$ to be $U\circ V=\{(x,z)|~(x,y)\in U~and~(y,z)\in V~ \mathrm{for~ some}~y\in X\}$. 
	\end{enumerate}
\end{Definition}

\begin{Definition}
	 A \textbf{coarse structure} $\mathcal{C}$ on a set $X$ is a family $\mathcal{X}$ of subsets of $X\times X$ that satisfy:
	\begin{enumerate}
		\item $\Delta\in\mathcal{X}$
		\item If $U\in\mathcal{X}$, then $U^{-1}\in\mathcal{X}$.
		\item If $U,V\in\mathcal{X}$, then $U\circ V\in\mathcal{X}$.
		\item If $U\in\mathcal{X}$ and $V\subseteq U$, then $V\in\mathcal{X}$.
		\item If $U,V\in\mathcal{X}$, then $U\bigcup V\in\mathcal{X}$.
	\end{enumerate}
		The elements of a large scale structure are called \textbf{controlled sets} or \textbf{entourages} and the pair $(X,\mathcal{C})$ is called a \textbf{coarse space}.
\end{Definition}

The interested reader in coarse spaces can look at \cite{Roe03}. We now provide definitions of a uniformly locally finite coarse space and
 of property A for a uniformly locally finite coarse space as defined by Sako. They can be found in \cite{Sako}:
\begin{Definition}
A coarse space $(X,\mathcal{C})$ is said to be \textbf{uniformly locally finite} if for every controlled set $T\in \mathcal{C}$ satisfies the inequality $sup_{x\in X}|T[x]|< \infty$ where\\ $T[x]=\{y\in X~|~(y,x)\in T\}$.
\end{Definition}
When giving the following definition, Sako states that a metric space with a uniformly locally finite coarse structure is called a metric space with bounded geometry.
\begin{Definition}(Sako Definition) 
A uniformly locally finite coarse space $(X,\mathcal{C})$ is said to have \textbf{property A} if for every positive number $\epsilon$ and every controlled set $T\in \mathcal{C}$, there exists a controlled set $S\in \mathcal{C}$ and a subset $A^{\mathcal{C}} \subset S \times \mathds{N}$ such that
\begin{itemize}
\item For $x \in X$, $A_x^{\mathcal{C}} = \{(y,n)\in X\times \mathds{N}; (x,y,n)\in A^{\mathcal{C}}\}$ is finite
\item $\Delta_X \times \{1\} \subset A^{\mathcal{C}}$, where $\Delta_X$ is the diagonal subset of $X\times X$.
\item $|A_x^{\mathcal{C}} \Delta A_y^{\mathcal{C}}| < \epsilon |A_x^{\mathcal{C}} \cap A_y^{\mathcal{C}}|$ if $(x,y)\in T$.
\end{itemize}
\end{Definition}

Dydak and Hoffland showed in \cite{Dydak} that there is a one to one correspondence between Coarse structures and Large Scale structures. The following are propositions 2.4 and 2.5 of \cite{Dydak}:

\begin{prop}
Let $X$ be a set. Every large scale structure $\mathcal{LSS}_X$ induces a coarse structure $\mathcal{C}$ on $X$ as follows: A subset $E$ of $X\times X$ is declared controlled if and only if there is a $\mathcal{B}\in\mathcal{LSS}_X$ such that $E\subset\bigcup\limits_{B\in\mathcal{B}} B\times B$.\\
Every coarse structure $\mathcal{C}$ on $X$ induces a large scale structure $\mathcal{LSS}$ on $X$ as follows: $\mathcal{B}$ is declared uniformly bounded if and only if there is a controlled set $E$ such that $\bigcup\limits_{B\in\mathcal{B}} B\times B\subset E$.
\end{prop}

Given the above proposition, an equivalent definition of property A for Large Scale structures with bounded geometry can be made. We provide our definition below to refresh the reader:

\begin{Definition} 
Let $X$ be a set and let $\mathcal{LSS}$ be a Large Scale structure on $X$ with bounded geometry. ($X, \mathcal{LSS}$) is said to have \textbf{property A} if for every $\epsilon >0$ and every uniformly bounded family $\mathcal{U} \in \mathcal{LSS}$, there exists $\mathcal{V} \in \mathcal{LSS}$ and a family of finite subsets $\{A_x^{\mathcal{LSS}}\}$ of $X\times \mathds{N}$ such that
\begin{itemize}
\item $A_x^{\mathcal{LSS}} \subset st(x,\mathcal{V})\times \mathds{N}$
\item $(x,1) \in A_x^{\mathcal{LSS}}$
\item $|A_x^{\mathcal{LSS}} \Delta A_y^{\mathcal{LSS}}| < \epsilon |A_x^{\mathcal{SLS}} \cap A_y^{\mathcal{LSS}}|$ if $y \in st(x,\mathcal{U})$.
\end{itemize}
\end{Definition}

We end the appendix with the claim that our definition of property A for bounded geometry is equivalent to Sako's definition of property A. Earlier in the paper, this was theorem 3.1:

\begin{Theorem}
 Let $(X,\mathcal{LSS})$ be a large scale space with bounded geometry. A uniformly locally finite coarse space $(X,\mathcal{C})$ has property A, then the induced large scale structure ($X, \mathcal{LSS}$) has property A. Likewise, if ($X, \mathcal{LSS}$) has property A, then the induced coarse structure $(X,\mathcal{C})$ also has property A.
\end{Theorem}
\begin{proof} 
Let $(X,\mathcal{LSS})$ be a large scale space with property A. Let $\mathcal{C}$ be the coarse structure induced by $\mathcal{LSS}$. Let $\epsilon>0$ and $T\in \mathcal{C}$. Thus, there exists a uniformly bounded family $\mathcal{U} \in \mathcal{LSS}$ with $T \subset \bigcup\limits_{U \in \mathcal{U}}U\times U$. Let $\mathcal{V}$ be the uniformly bounded family satisfying the definition of property A for the above $\epsilon$ and $\mathcal{U}$. Define $S:= \bigcup\limits_{V \in \mathcal{V}}V\times V$ which is a controlled set by prop 4.2. Define $A^{\mathcal{C}}:= \bigcup\limits_{x\in X} \{x\}\times A_x^{\mathcal{LSS}}$.

We will show that $A^{\mathcal{C}}\subset S\times \mathds{N}$. Given $ (x,y,n)\in A^{\mathcal{C}}$ then $(x,y)\in \{x\}\times A_x^{\mathcal{LSS}}$ or $y \in A_x^{\mathcal{LSS}} \subset st(x,\mathcal{V})$. Thus, there exists a $V\in \mathcal{V}$ such that $x,y\in V$. Therefore $(x,y)\in V\times V$, and $A^{\mathcal{C}}\subset S\times \mathds{N}$.

Next, we will show that for $x \in X$, $A_x^{\mathcal{C}} = \{(y,n)\in X\times \mathds{N}; (x,y,n)\in A\}=A_x^{\mathcal{LSS}}$. Let $x\in X$ and consider $(y,n)\in A_x^{\mathcal{C}}$. By definition $(y,n)\in A_x^{\mathcal{LSS}}$. Similarly, if $(z,m)\in A_x^{\mathcal{LSS}}$, then $(x,z,m)\in A^{\mathcal{C}}$, and therefore $(z,m)\in A_x^{\mathcal{C}}$, so $A_x^{\mathcal{C}}= A_x^{\mathcal{LSS}}$ \\
By definition of $A_x^{\mathcal{LSS}}$, and $A^{\mathcal{C}}$ we get
\begin{itemize}
\item For $x \in X$, $A_x^{\mathcal{C}} = \{(y,n)\in X\times \mathds{N}| (x,y,n)\in A^{\mathcal{C}}\}$ is finite
\item $\Delta_X \times \{1\} \subset A^{\mathcal{C}}$, where $\Delta_X$ is the diagonal subset of $X\times X$.
\end{itemize}
The third property also follows easily as if $(x,y)\in T$, then there exists $U \in \mathcal{U}$ with $x,y \in U$. Thus, $y\in st(x,\mathcal{U})$, so we have
$$\frac{|A_x^{\mathcal{C}} \Delta A_y^{\mathcal{C}}|}{|A_x^{\mathcal{C}} \cap A_y^{\mathcal{C}}|} = \frac{|A_x^{\mathcal{LSS}} \Delta A_y^{\mathcal{LSS}}|}{|A_x^{\mathcal{LSS}} \cap A_y^{\mathcal{LSS}}|}< \epsilon$$
and $(X,\mathcal{C})$ has property A. \\
Conversly, assume that $(X,\mathcal{C})$ has Property A. Then we will show that $(X,\mathcal{LSS})$ has property A, where $\mathcal{LSS}$ is the large scale structure induced by $\mathcal{C}$. Let $\epsilon >0$ and $\mathcal{U}$ be a uniformly bounded cover. Thus there exists a controlled set $T\in \mathcal{C}$ such that $\bigcup\limits_{U\in \mathcal{U}}U\times U \subset T$. Let $S$ be the controlled set and $A\subset S\times \mathds{N}$ satisfying property A for $\epsilon$ and $T$. \\
Define $\mathcal{V}=\{V_x\}_{x\in X}$ where $V_x = \left\{y\in X; (y,n)\in A_x, n\in \mathds{N}\right\}$. It needs to be shown that $\mathcal{V}$ is uniformly bounded: given $(y,z)\in \bigcup\limits_{V\in \mathcal{V}}V\times V$. Then, for some $x \in X, (x,y,n), (x,z,m)\in A^{\mathcal{C}}\subset S\times \mathds{N}$ for some $n,m \in \mathds{N}$. In addition, $(x,y), (x,z) \in S$ meaning $(y,x)\in S^{-1}$ and thus $(y,z)\in S^{-1}\circ S$ which is a controlled set. Since $(y,z)$ was chosen arbitrarily, we have that $\mathcal{V}$ is uniformly bounded. \\
Note that $st(x,\mathcal{V})=\{y\in X; (y,n)\in A_z^{\mathcal{C}}$ and $(x,m)\in A_z^{\mathcal{C}}$ for $z \in X\}$. Since $(x,1)\in A_x^{\mathcal{C}}$, then $A_x^{\mathcal{C}}\subset st(x,\mathcal{V})$. Setting $A_x^{\mathcal{LSS}}=A_x^{\mathcal{C}}$ one can show:
\begin{itemize}
\item $A_x^{\mathcal{LSS}} \subset st(x,\mathcal{V})\times \mathds{N}$
\item $(x,1) \in A_x^{\mathcal{LSS}}$
\end{itemize}
are true. Finally, if $y \in st(x,\mathcal{U})$, then there exists $U \in \mathcal{U}$ such that $x,y\in U$ which means $(x,y)\in \bigcup\limits_{U\in \mathcal{U}}U\times U \subset T$ and the following holds:
$$\frac{|A_x^{\mathcal{LSS}} \Delta A_y^{\mathcal{LSS}}|}{|A_x^{\mathcal{LSS}} \cap A_y^{\mathcal{LSS}}|} = \frac{|A_x^{\mathcal{C}} \Delta A_y^{\mathcal{C}}|}{|A_x^{\mathcal{C}} \cap A_y^{\mathcal{C}}|} < \epsilon.$$
Thus $(X,\mathcal{LSS})$ has property A.
\end{proof}

\end{document}